\documentclass[12pt,reqno]{amsart}
\usepackage {amsmath,amssymb,verbatim}
\usepackage[pagebackref,pdftex,hyperindex]{hyperref}
\usepackage[all]{xy}

\usepackage[margin=3cm]{geometry}

\newcommand{\C}{\mathbb{C}}

\renewcommand{\P}{\mathbb{P}}
\newcommand{\Q}{\mathbb{Q}}
\newcommand{\R}{\mathbb{R}}
\renewcommand{\S}{\mathbb{S}}
\newcommand{\T}{\mathbb{T}}

\newcommand{\Z}{\mathbb{Z}}

\newcommand{\fg}{\mathfrak{g}}
\newcommand{\fh}{\mathfrak{h}}

\newcommand{\ft}{\mathfrak{t}}

\newcommand{\cH}{\mathcal{H}}

\newcommand{\cJ}{\mathcal{J}}

\newcommand{\cL}{\mathcal{L}}

\newcommand{\cO}{\mathcal{O}}

\newcommand{\cX}{\mathcal{X}}

\renewcommand{\a}{\alpha}

\renewcommand{\d}{\delta}

\renewcommand{\l}{\lambda}

\renewcommand{\phi}{\varphi}

\renewcommand{\i}{\sqrt{-1}}
\renewcommand{\leq}{\leqslant}
\renewcommand{\geq}{\geqslant}
\newcommand{\abs}[1]{\left\lvert#1\right\rvert}
\newcommand{\norm}[1]{\left\|#1\right\|} 
\newcommand{\inn}[1]{\left\langle#1\right\rangle}

\newcommand{\FS}{\mathrm{FS}}

\newcommand{\PS}{\mathrm{PS}}
\newcommand{\pt}{\mathrm{pt}}
\renewcommand{\Re}{\mathrm{Re}}

\DeclareMathOperator{\Aut}{Aut}
\DeclareMathOperator{\aut}{\mathfrak{aut}}

\DeclareMathOperator{\diag}{diag}

\DeclareMathOperator{\DF}{DF}

\DeclareMathOperator{\GL}{GL}
\DeclareMathOperator{\fgl}{\mathfrak{g}\mathfrak{l}}

\DeclareMathOperator{\Ham}{Ham}

\DeclareMathOperator{\id}{id}

\DeclareMathOperator{\fsl}{\mathfrak{s}\mathfrak{l}}

\DeclareMathOperator{\Tr}{Tr}

\numberwithin{equation}{section}       

\newtheorem{prop} {Proposition} [section]
\newtheorem{thm}[prop] {Theorem} 
\newtheorem{dfn}[prop] {Definition}
\newtheorem{lem}[prop] {Lemma}

\newtheorem{exam}[prop]{Example}
\newtheorem{rem}[prop]{Remark}
\theoremstyle{remark}
\newtheorem*{ackn}{\bf{Acknowledgment}}

\newtheorem*{thmA}{\bf{Theorem A}} 
\newtheorem*{thmB}{\bf{Theorem B}} 
 
\newtheorem*{thmD}{\bf{Theorem D}}

\newtheorem*{corC}{\bf{Corollary C}} 

\newtheorem*{dfn*}{\bf{Definition}}

\title[Orthogonal projection of a test configuration to vector fields]{Orthogonal projection of a test configuration to vector fields} 
\date{\today} 

\author{Tomoyuki Hisamoto}
\address{Graduate School of Mathematics\\
  Nagoya University\\
  Furocho\\
  Chikusa\\
  Nagoya\\ 
  Japan}
\email{hisamoto@math.nagoya-u.ac.jp}


\begin{document}

\maketitle
\setcounter{tocdepth}{1}

\begin{abstract}
Given a polarized complex manifold, projection of a torus-equivariant test configuration to holomorphic vector fields was introduced by \cite{Sze06}, as the limit of the associated $\mathbb{C}^*$-actions. We show that there actually holds the moment convergence of the weight distributions. Our analytic approach at the same time specifies the limit in terms of the weak geodesic ray associated with the test configuration. 
Related to the result, we discuss about the reduced $L^p$-norm of the test configuration in attempt to describe the uniform K-stability of the polarization relative to the automorphism group. 
\end{abstract} 

\tableofcontents 

\section{Introduction}\label{introduction} 

Let $(X, L)$ be an $n$-dimensional complex polarized manifold. Our background is the space of K\"ahler metrics which are in the fixed first Chern class $c_1(L)$. By $dd^c$-lemma, it is equivalent to consider the collection of all fiber metrics $e^{-2\phi}$ with positive curvature $\omega_{\phi}:= dd^c\phi$, which we denote by $\cH$. Tangent vectors on $\phi\in \cH$ are identified with a smooth function on $X$ so that the integration by the Monge-Amp\`ere measure 
\begin{equation*}
\inn{u, v}_{\phi}:= \int_X u v  \omega_{\phi}^n/n!
\end{equation*}
defines a natural Riemannian structure on $\cH$. Algebraically a geodesic ray should be regarded as a $\C^*$-equivariant flat family of $\Q$-polarized schmes $(\cX, \cL) \to \C$ with generic fiber isomorphic to $(X, L)$. Such a degeneration called test configuration is important to study various stabilities in classical Geometric Invariant Theory and it also plays a roll testing K-stability of a polarized manifold. Indeed given a test configuration with $\cX$ normal one can attach a {\em weak} geodesic ray $\phi^t$ as a pull-back (by the action of $e^{t}$) of the unique solution $\Phi$ for the Monge-Amp\`ere equation. That is a local fiber metric $e^{-2\Phi}$ of $\cL$ satisfying 
\begin{equation*}
(dd^c \Phi)^{n+1} =0  \ \ \  \text{on} \ \ \ \cX_{\abs{\tau} \leq 1}, 
\end{equation*}
with a Dirichlet boundary value which corresponds to the initial point. 

The most special case is a product test configuration, which is a product family with $\C^*$-action generated by a  holomorphic vector field $v \in \fh(X, L)$. In fact, in parallel to the Hilbert-Mumford criterion in GIT Donaldson defined K-stability generalizing the classical Futaki character defined over $\fh(X, L)$ to the Donaldson-Futaki invariant $\DF(\cX, \cL)$ attached to each test configuration. 

At this point we are led to study ``angle" between a test configuration and a vector field by the above inner product. The algebraic formulation was already introduced by \cite{Sze06} in order to define his uniform K-stability. Let us review it briefly. 
For the purpose it was revealed natural to restrict ourselves to equivariant test configurations. 
Fix a sub-torus $T \subseteq \Aut^0(X, L)$ and denote the reduced Lie algebra by $\ft/\C  \subseteq \fh(X, L)$. Let $(\cX, \cL)$ be a $T$-equivariant test configuration. For any sufficiently divisible $k\geq 1$, the induced action $\l^{(k)}$ on the central fiber $(\cX_0,, \cL_0^{\otimes k})$ defines $A_k \in \fsl(H^0(\cX_0,, \cL_0^{\otimes k}))$ as the trace-free part of the generator. On the other hand $T$-equivariance defines a faithful Lie-algebra representation $\ft/\C \to \fsl(H^0(\cX_0,, \cL_0^{\otimes k})^\vee)$. Under the identification of $\ft/\C$ and its image of the dual representation we may consider the canonical Killing form which yields the orthogonal projection 
\begin{equation*}
P_k: \fsl(H^0(\cX_0,, \cL_0^{\otimes k})) \to (\ft/\C)_\R. 
\end{equation*}
For a fixed basis, coefficients of $P_k(A_k)$ defines a rational function in $k$, which converges to some rational number, as one can see by the equivariant version of the Riemann-Roch-Hirzebruch theorem. Therefore $P_k(A_k)$ converges to an element of $(\ft/\C)_\Q$. It is natural to denote it as $P(\cX, \cL)$. The definition of $P(\cX, \cL)$ already appeared in \cite{CS16}. 

The purpose of this paper is to give the analytic description of $P(\cX, \cL)$ and to show that more precise convergence holds, supporting Sz\'ekelyhidi's insight. 
For a fixed K\"aher form $\omega=dd^c\phi^0$ any vector field $v\in \ft/\C$ defines a unique complex Hamilton function characterized by the property $ \i\bar{\partial}u = i_v \omega$ with the normalization $\int_X u \omega^n=0$. We will take the reference metric $\omega$ as $S:=\ft_\R\otimes \S^1$ invariant. It then follows that the normalized complex Hamilton function takes real values for any $v \in  \ft/\C$. 
We denote the $\R$-vector space consists of such functions as $\Ham_0(\ft, \omega)$. 
It is naturally contained in the space of normalized square integrable functions $L^2_0(X, \omega^n)$. Given a geodesic ray $\phi^t$ emanating from $\phi^0$, the orthogonal projection 
\begin{equation*}
P: L^2_0(X, \omega^n) \to \Ham_0(\ft, \omega)
\end{equation*}
defines the ``projection of the tangent vector" $P(\dot{\phi}^0)$.

\begin{thmA}
Let $(\cX, \cL)$  be a $T$-equivariant normal test configuration. Take a weak geodesic ray $\phi^t$ associated with $(\cX, \cL)$ and an initial metric $\omega= dd^c \phi^0$ which is $S$-invariant. Then, for any integer $p \geq 1$ the moment convergence 
\begin{equation*}
\lim_{k \to \infty} \frac{\Tr P_k(A_k)^p}{k^pN_k} = \frac{1}{V}\int_X P(\dot{\phi}^0)^p \omega^n  
\end{equation*}
holds. 
\end{thmA}

When the test configuration is generated by a vector field, the formula is a consequence of equivariant Chern-Weil theory as it was explained by \cite{Sze06}. The above theorem gives a generalization of \cite{His12} where $T=\C^*$ case was treated. The argument in the present paper gives another proof for that theorem. 

Let us explain for our related discussion about K-stability. 
Setting 
\begin{equation*}
\norm{(\cX, \cL)}_p = \bigg( \frac{1}{V}\int_X \abs{\dot{\phi}^0}^p \omega^n  \bigg)^{1/p}, 
\end{equation*}
we say a polarized manifold is {\em uniformly} K-stable if there exists a positive constant $\d$ such that 
\begin{equation*}
\DF(\cX, \cL) \geq \d \norm{(\cX, \cL)}_1  
\end{equation*}
holds for any normal test configuration. The notion of uniform K-stability (relative to a maximal torus) was first introduced \cite{Sze06}. 
We have to take $p\leq \frac{n}{n-1}$ otherwise there is no uniformly K-stable polarized manifold for $n \geq 3$ (see \cite{Sze06} and \cite{BHJ15}). 
The exponent $p=1$ has speciality here in relation with the coercivity property of the K-energy.  
The theory of such new stability was developed as one can seen in \cite{Sze06}, \cite{Der14a}, \cite{Der14b}, \cite{BBJ15}, \cite{BHJ15}, \cite{DR15}, \cite{BHJ16}, and \cite{BDL16}. 
Especially it was shown by \cite{BDL16} that if the reduced automorphism group is discrete, any polarized manifold $(X, L)$ admitting a constant scalar curvature K\"ahler metric is uniformly K-stable. 
However, when $\Aut^0(X, L)/\C^*$ is not discrete the  above norm does not work to specify the existence of a special metric. 
The situation lets us analytically define the reduced norm as follows. 
Take a torus and set $P^\bot:= \id -P$ as the projection to the orthogonal complement. The reduced norm for $T$-equivariant test configuration is 
\begin{equation*}
\norm{(\cX, \cL)}_{T, p} := \bigg( \frac{1}{V}\int_X \abs{P^\bot(\dot{\phi}^0)}^p \omega^n  \bigg)^{1/p}. 
\end{equation*} 
Of course by Theorem A it coincides with the algebraic definition of \cite{Sze06} and \cite{CS16}. 
Now we take the torus as the center $C(G)$ of a subgroup $G\subseteq \Aut^0(X, L)$. We say that a polarized manifold is uniformly K-stable relative to $G$ if there exists a constant $\d>0$ such that 
\begin{equation*}
\DF(\cX, \cL) \geq \d \norm{(\cX, \cL)}_{C(G), 1}   
\end{equation*} 
holds for any $G$-equivariant test configuration $(\cX, \cL)$. 
At this stage there is no natural choice of a maximal torus and we are especially interested in the case $G=\Aut^0(X, L)$, which should be related with the existence of a special metric. 
See also the explanation below Definition \ref{uniform stability relative to G}. 
The restriction to the $G$-equivariant test configurations seems a bit strong but admissible. It is because if there exists a constant scalar curvature K\"ahler metric $G=\Aut^0(X, L)$ is reductive by the Matsushima-Lichnerowicz theorem and then the Donaldson-Futaki invariant vanishes on the semisimple part $[\fg/\C, \fg/\C]$ of $\fg/\C$ since it defines a character. 
The following result we will show supports the above definition of the norm for it gives a characterization of product test configurations. 
\begin{thmB}
Let $T \subseteq \Aut^0(X, L)$ be an algebraic torus. Then a $T$-equivariant normal test configuration $(\cX, \cL)$ is product if and only if the reduced norm $\norm{(\cX, \cL)}_{T, p}$ is zero for some $p\geq 1$. 
\end{thmB} 
When $p=2$ the result was independently proved in \cite{CS16}. In fact once noting that the projection $P(\phi^0)$ defines a rational element we may easily reduce it to $\fh(X, L)=\{0\}$ case where the result was obtained by\cite{BHJ15} and \cite{Der14a}. 
The reduced norm is determined by the $\C^*$-action on the central fiber so that it immediately yields: 
\begin{corC} 
A $T$-equivariant test configuration is product coming from $\ft$ if and only if the central fiber $(\cX_0, \cL_0)$ is $T$-equivariantly isomorphic to $(X, L)$ endowed with some $\C^*$-action to $T$.  
\end{corC}
Notice when $T$ is trivial we assumed that the induced action on the central fiber is trivial. It seems not known whether a test configuration whose central fiber is only isomorphic to $(X, L)$ is product. 
Finally, we remark that our study combined with the argument of \cite{His12} immediately deduces the relative version of lower bound estimate for the Calabi type functionals, at least in the Fano case. We recall that the Donaldson-Futaki invariant relative to $G$ is defined to be 
\begin{equation*}
\DF_T(\cX, \cL):= \DF(\cX, \cL)- \DF(P(\cX, \cL)). 
\end{equation*}
\begin{thmD}
Let $X$ be a Fano manifold and $\omega$ a K\"ahler metric in the anti-canonical class. We denote the normalized Ricci potential by $h_\omega$. Then for any normal test configuration $(\cX, \cL)$ we have the lower bound estimate 
\begin{equation*}
\bigg(\frac{1}{V}\int_X \abs{P^{\bot}({e^{h_{\omega}}-1})}^q \omega^n\bigg)^{1/q}
\geq \frac{-\DF_T(\cX, \cL)}{\norm{(\cX, \cL)}_{T, p}}. 
\end{equation*}

\end{thmD}

\begin{ackn}
The author express his gratitude to Yuji Sano for his indicating the original idea of \cite{FM95}. We discussed during our stay in Banff center where the workshop ``Complex Monge-Amp\`{e}re Equations on Compact K\"{a}hler Manifolds" was held in the spring, 2014. We are grateful to their support. The author also wishes to thank Robert Berman, S\'ebastien Boucksom, Mattias Jonsson, Yasufumi Nitta, Shinnosuke Okawa, and David Witt Nystr\"om for their helpful comments. Especially in communication with Berman and Jonsson we were made to realize the necessity of G-equivariance in our formalism and their equivarent (but seems to be more direct) definition of the reduced norm. The author is supported by JSPS KAKENHI Grant Number 15H06262. 
\end{ackn}


\section{Preliminary materials}\label{preliminary}

\subsection{Automorphism group, Hamilton function}

We starts from reviewing the basic materials for the automorphism group and Hamiltonnian action. It seems that involved definitions are sometimes confused in literatures so I believe it is helpful to organize things here. In this paper $X$ always stands for a smooth complex projective variety and $L$ is an ample line bundle on $X$. We denote the automorphism group of $L$ by $\Aut(X, L)$ and mainly focus on the identity component $\Aut^0(X, L)$. The reduced automorphism group $\Aut(X, L)/\C^*$ is embedded to $\Aut(X)$: the automorphism group of $X$. For any $k \geq 1$ the automorphism group $\Aut(X, L)$ naturally acts on $H^0(X, L^{\otimes k})$ so that $g \in \Aut(X, L)$ shifts $s \in H^0(X, L^{\otimes k})$ in the manner:  
\begin{equation}\label{action}
(g^*s)(x) := (g^{\otimes k})^{-1}(s(gx)),  
\end{equation}
where $g^{\otimes k}: L^{\otimes k} \to L^{\otimes k}$ is the induced automorphism. It defines the canonical representation $\Aut(X, L) \to \GL(H^0(X, L^{\otimes k})^{\vee})$ which is faithful for $k\gg 1$ by the ampleness. It particularly shows that $\Aut(X, L)$ is an algebraic group ($\Aut(X)$ is not!). 

The reduced Lie algebra 
\begin{equation*}
\fh(X, L) := \aut(X, L)/\C
\end{equation*}
consists of holomorphic vector fields which can be lifted to $L$. Once a K\"ahler form $\omega \in c_1(L)$ is fixed, the lifting property for a vector field $v \in \fh(X)$ is equivalent to the existence of a complex Hamilton function $u$ which is characterized by 
\begin{equation*}
\i \bar{\partial} u = i_v \omega. 
\end{equation*}
We refer \cite{Don02} for the proof. Just like $v$, $u$ is unique up to the constant. We say $u$ is normalized if it satisfies $\int_X u \omega^n=0$. Notice that the above characterization in particular implies $\fh(X, L^{\otimes k})=\fh(X, L)$ and hence $\Aut^0(X, L^{\otimes k})=\Aut^0(X, L)$ for any exponent $k \geq 1$. 

Let us consider a sub-torus $T\subseteq \Aut^0(X, L)$. Collection of all one-parameter subgroups (1-PS for short) to $T$ naturally defines a lattice $\ft_\Z$ in $\ft$ so that $\ft_\R:= \ft_\Z \otimes \R$ defines a natural real form. We denote the compact part by $S:=\ft_\Z \otimes \S^1$ and take $\omega$ as $S$-invariant.  
Then $\i\partial \bar{\partial}u= L_{\Re v} \omega$ holds for a general $v \in \ft/\C$, since the Cartan's formula yields 
\begin{equation*}
 \i\partial \bar{\partial}u  = d(\sqrt{-1}\bar{\partial}u) =  di_v \omega = L_v \omega=L_{\Re v}\omega.  
\end{equation*} 
One can observe that $u$ is a real function modulo the constant because the  right-hand side is a real form. 

Finally, the differential representation $\aut(X, L) \to \fgl(H^0(X, L^{\otimes k})^\vee)$ induces 
\begin{equation*}
\fh(X, L) \to \fsl(H^0(X, L^{\otimes k})^\vee). 
\end{equation*}  
This map is the starting point of our study. 
For each vector $v \in \aut(X, L)$ the image of $-v$ by the dual representation $\aut(X, L) \to \fgl(H^0(X, L^{\otimes k}))$ is called the generator.     
In the next subsection we extend the formalism to general equivariant test configurations. Anyhow, it essentially requires us to consider bundle morphisms $\Aut(X, L)$ while $\fh(X, L)$ is adopted to our purpose thanks to the non-degenerate Killing form .


\subsection{Test configurations and equivariant embeddings}

We shall introduce a test configuration $(\cX, \cL)$ as a $\C^*$-equivariant flat family of $\Q$-polarized schemes, with generic fiber isomorphic to $(X, L)$. 
In this note the total space $\cX$ is always assumed to be normal, while the central fiber is very singular, not even reduced. As it was shown by \cite{LX11}, the normality assumption is in fact necessary to exclude a pathological example for which conclusion of Theorem B does not hold. 

We denote the $\C^*$-action by $\l$. The induced action $\l^\vee: \C^* \to \GL(H^0(\cX_0, \cL_0^{\otimes k}))$ defines $A'_k \in \ \fgl(H^0(\cX_0, \cL_0^{\otimes k}))$ such that $\l^\vee(e^{-t}) = \exp(tA'_k)$ holds for any $t \in \R$. We call it a generator. As a standard fact of representation theory  $A'_k$ is diagonalized to $\diag (\l_1', \dots \l_{N_k}')$ in a chosen basis. Following \cite{Don02}, we define the Donaldson-Futaki invariant as the coefficient in the subleading term of the asymptotic expansion
\begin{equation*}
\frac{\sum_{i=1}^{N_k} \l_i'}{kN_k} = F_0 + F_1 k^{-1} + O(k^{-2}). 
\end{equation*}
Namely, $\DF(\cX, \cL):= F_1$. In the present note, however, we rather mind the trace-free part 
\begin{equation*}
A_k := A_k' - \frac{\Tr A_k}{N_k} \id 
\end{equation*}
with weights $\l_i = \l_i' - N_k^{-1}\sum \l_i'$ and focus on the limit behavior of $\Tr A_k^p$. 

The most basic examples are product test configurations. We call $(\cX, \cL)$ product if $\cX \simeq X \times \C$ and $\cL \simeq p_1^*L$ holds. It is equivalent to give a one-parameter subgroup (1-PS in short) $\l: \C^* \to \Aut^0(X, L)$, which uniquely defines a $\C^*$-action on the product space in the equivariant way. A general test configuration can be regarded as a natural generalization of such a 1-PS as one takes an equivariant embedding. Actually, for any $r \geq 1$ any 1-PS $\l^\vee: \C^* \to \GL(H^0(X, L^{\otimes r}))$ and its dual $\l=(\l^\vee)^\vee$ defines a test configuration as the Zariski closure of the orbit: 
\begin{equation*}
\cX := \overline{\cup_{\tau \in\C^*} \l(\tau)(X) \times \{\tau \}} \subseteq \P(H^0(X, L^{\otimes r})^\vee)\times \C, 
\end{equation*} 
endowed with the $\Q$-line bundle $\cL:= \cO(1/r)\big\vert_{\cX}$. If one forcuses on the central fiber one isomorphism $H^0(\cX_0, \cL_0^{\otimes k}) \simeq H^0(X, L^{\otimes k})$ is obtained from this construction. 
Let us take a sufficiently divisible $k\geq 1$ and an equivariant embedding 
\begin{equation*}
I_k: (\cX, \cL^{\otimes k})  \to (\P(H^0(\cX_0, \cL_0^{\otimes k})^\vee)\times \C, \cO(1))
\end{equation*}
so that the restriction on the central fiber gives the Kodaira embedding. The embedding $I_k$ is determined by the isomorphism $T_k: H^0(\cX_0, \cL_0^{\otimes k}) \to H^0(X, L^{\otimes k})$ which makes the following diagram commutes. 
\[\xymatrix{
   {H^0(\cX_0, \cL_0^{\otimes k})}\ar[r]^{I_k^*}\ar[d]_{\l^\vee(\tau)} 
 & {H^0(\cX_\tau, \cL_\tau^{\otimes k})}\ar[d]^{\l^\vee(\tau)} \\ 
   {H^0(\cX_0, \cL_0^{\otimes k})}\ar[r]_{T_k} 
 & H^0(\cX_1, \cL_1^{\otimes k})
}
\]
Such an isomorphism $T_k$ derived from an equivariant embedding is called regular. 
Notice that the flatness assumption yields $H^0(\cX_0, \cL_0^{\otimes k}) \simeq H^0(X, L^{\otimes k})$ but we do not have a natural isomorphism. 
As a conlusion, equivariantly embedded test configuration is equivalent to 1-PS $\l: \C^* \to \GL(H^0(X, L^{\otimes r})^\vee)$. 
If moreover the generator $A_k$ is Hermitian with respect to the pull-back of a Hermitian inner product $H_k$ on $H^0(X, L^{\otimes k})$, we call $T_k$ Hermitian. For us the following existence theorem is important. 

\begin{thm}[\cite{Don05}]\label{equivariant embedding}
Given a test configuration and a Hermitian inner product $H_k$ on $H^0(X, L^{\otimes k})$, there exists a regular Hermitian isomorphism $T_k: H^0(\cX_0, \cL_0^{\otimes k}) \to H^0(X, L^{\otimes k})$. Moreover, the pull-back $T_k^*H_k$ is independent of $T_k$. For two regular Hermitian isomorphisms $T_k$ and $T_k'$ there exists a unitary endomorphism $U$ commutative with the generator $A_k$ such that $T_k' =T_k \circ U$ holds. 
\end{thm}

\subsection{Associated weak geodesic ray}

Here we briefly review the construction of the weak geodesic ray from a test configuration. Let $(\cX, \cL)$ be a test configuration with $\cX$ normal. We rotate the reference metric $\phi$ by the $\C^*$-action to adopt it as a boundary value of the Monge-Amp\`ere equation: 
\begin{equation}\label{MA}
(dd^c \Phi)^{n+1}=0 \ \ \ \text{on} \ \ \  \cX_{\abs{\tau}\leq 1}. 
\end{equation}
In \cite{Berm12} the $L^{\infty}$-solution was simply constructed as the Perron envelope of subsolutions. Regarding the computation of \cite{Sem92}, the associated $\phi^t(x):= \Phi(\l(e^t)x)$ $t \in [-\infty, 0]$ defines a weak geodesic ray emanating from $\phi$, where we have the canonical Mabuchi metric on the tangent space at $\phi$. It coincides with the preceding construction of Phong and Sturm. 

\begin{thm}[\cite{PS07a}, \cite{PS07b}, and \cite{PS10}]\label{PS}
Let $(\cX, \cL)$ be a test configuration and $\phi$ a reference metric. For any sufficiently divisible $k \geq 1$ there exists an equivariant embedding 
\begin{equation*}
I_k: (\cX, \cL^{\otimes k})  \to (\P(H^0(\cX_0, \cL_0^{\otimes k})^\vee)\times \C, \cO(1))
\end{equation*}
whose restriction on the central fiber gives the Kodaira embedding, and a basis $\{s_i\}_{i=1}^{N_k} \subset H^0(X, L^{\otimes k})$ such that the followings hold. 
\begin{itemize}
\item[$(i)$]
The basis is orthonormal for the $L^2$-Hermitian norm
\begin{equation*}
\norm{s}_{k\phi}^2 := \int_X \abs{s}^2 e^{-2k\phi} \omega_{\phi}^n. 
\end{equation*}
\item[$(ii)$]
The associated regular isomorphism $T_k: H^0(\cX_0, \cL_0^{\otimes k}) \to H^0(X, L^{\otimes k})$ is Hermitian in $\norm{\cdot }_{k\phi}^2$ and the trace-free part of the generator $A_k$ is diagonalized to $\diag(\l_1, \dots \l_{N_k})$ by the basis $\{T_k^*s_i\}_{i=1}^{N_k}$. 
\item[$(ii)$] 
Setting 
\begin{equation*}
\phi_k^t:= \frac{1}{2k}\log \sum_{i=1}^{N_k}e^{2\l_i t}\abs{s_i}^2, 
\end{equation*}
The upper-semicontinuous regularized limit of the supremum 
\begin{equation*}
\phi^t := (\lim_{r \to \infty} \sup_{k \geq r} \phi_k^t)^*
\end{equation*}
defines a weak geodesic ray of class $C^{1, \alpha}$ for any $0<\alpha<1$. The initial point is $\phi^0=\phi$. 
\end{itemize}

\end{thm}
The claim $(i)$ and $(ii)$ are the direct consequences of Theorem \ref{equivariant embedding}. In fact the above construction shows that 
\begin{equation*}
\phi_k^t(x) = \Phi_{\FS, k} (\l(e^t)x) 
\end{equation*}
is the pull-back of a particularly chosen Fubini-Study weight on $\cL$. If for a while we denote $\Phi_{\PS}$ as the metric corresponding to $\phi^t$ of Theorem \ref{PS}, it is at least defined on $\cX_{\tau\neq 0}$. Let $\Phi_{\mathrm{B}}$ be the bounded solution of  (\ref{MA}). Since $\Phi_{\FS, k}$ are defined over $\cX_{\abs{\tau}\leq 1}$ the comparison principle for the Monge-Amp\`ere operator implies 
$\Phi_{\FS, k}\leq \Phi_{\mathrm{B}}$ hence $\Phi_{\PS}\leq \Phi_{\mathrm{B}}$. Therefore $\Phi_{\PS}$ is upper bounded so extended to a plurisubharmonic weight of $\cL$ , which satisfies the degenerated Monge-Amp\`ere equation (\ref{MA}). Uniqueness of the solution now deduces $\Phi_{\PS} = \Phi_{\mathrm{B}}$.


\section{Quantized orthogonal projection}\label{quantized orthogonal projection}

\subsection{Proof of the main theorem}\label{proof of the main theorem}

We fix a reference metric $\omega=dd^c\phi$ in $c_1(L)$ to consider the $L^2$-norm 
\begin{equation}\label{L2norm}
\norm{s}^2_{k\phi} := \int_X \abs{s}^2 e^{-2k\phi} \omega^n. 
\end{equation}
It defines a sequence of Hermitian inner product on the finite dimensional vector space quantization guideline Bergman construction  
\begin{equation}\label{bergman}
\phi_k := \frac{1}{2k} \log \sum_{i=1}^{N_k} \abs{s_i}^2. 
\end{equation}

Let us first examine the case which corresponds to a product test configuration. For the convenience we set $s=(s_1, \dots s_{N_k})$ and define $s^*$ to be the adjoint. 

\begin{prop}\label{product case}
Assume that $\omega$ is invariant for the compact torus $S:=\ft_{\R}\otimes \S^1$. 
Let $v \in (\ft/\C)_\R$ and denote the trace-free part of the generator for some lift by $B_k \in \fsl(H^0(X, L^{\otimes k}))$. Take any orthonormal basis $\{s_i\}_{i=1}^{N_k} \subset H^0(X, L^{\otimes k})$ for $\norm{\cdot}_{k\phi}$. Then, the normalized complex Hamilton function is described as 
\begin{equation*}
u = \frac{s^* B_k s}{ks^*s} +O(k^{-1}\log k). 
\end{equation*}
\end{prop} 
\begin{proof}
Since the identity is linear both in $u$ and $B_k$ we may assume that $v\in (\ft/\C)_\Q$, namely, it derives from a 1-PS $\mu: \C^* \to T$. The Hamiltonian property $\i \bar{\partial} u = i_v \omega$ deduces 
\begin{equation}\label{h}
u= \frac{d}{dt}\bigg\vert_{t=0} \exp (t v)^*\phi +c,  
\end{equation}
for some constant $c \in \R$. We need to note that $\phi$ is a weight only locally defined on some trivialization patch. If one takes the trivialization in an equivariant way, it can be checked that the weight of the pull-backed fiber metric is written as 
\begin{equation}\label{equivariant trivialization}
\exp (t v)^*\phi + \log \abs{\chi (e^t)}, 
\end{equation}
with a character $\chi$. By the differentiation, however, the transition rule is canceled and hence the right-hand side of (\ref{h}) is globally defined. Let us take the Bergman weight. 
By Demailly's approximation theorem, there exists a positive constant $C$ such that 
\begin{equation}\label{TYZ}
\phi_k -Ck^{-1}\log k \leq \phi \leq \phi_k + Ck^{-1} 
\end{equation}
holds for any $k \geq 1$. Noting (\ref{action}) and (\ref{equivariant trivialization}) we observe that the generator $B'_k$ satisfies  
\begin{equation*}
 \exp (t v)^*\phi_k = \mu(e^{t})^*\phi_k = \frac{1}{2k}\log \sum_{i=1}^{N_k} \abs{\exp(tB'_k)s_i}^2. 
\end{equation*}
The differentiation then yields  
\begin{equation*}
u-c =  \frac{s^* B'_k s}{ks^*s} +O(k^{-1}\log k). 
\end{equation*}
To determine the constant $c$, we choose the basis so that $\mu(\tau)s_i = \tau^{-\mu'_i}s_i$ holds. Then the trace-free part $B_k= \diag(\mu_1, \dots, \mu_{N_k})$ can be compared with $B'_k$ as 
\begin{equation*}
\frac{s^* B'_k s}{ks^*s} = \frac{\sum \mu'_i \abs{s_i}^2}{k\sum \abs{s_i}^2} = \frac{\sum \mu_i \abs{s_i}^2}{k\sum \abs{s_i}^2} +\frac{\Tr B'_k}{kN_k }.   
\end{equation*}
Thus the proof is concluded with 
\begin{equation*}
\frac{\Tr B'_k}{kN_k} \to \frac{1}{V}\int_X (u-c) \omega^n, 
\end{equation*}
which is a consequence of the equivariant Chern-Weil theory (\cite{AB84}). (It also follows from \cite{His12} combined with \cite{WN10}'s computation of the tangent for a product test configuration.) 
\end{proof} 

Let us go to the general case. Our proof is based on the following key lemma. 
\begin{lem}\label{Lp}
Let $\phi^t$ be a weak geodesic ray with $\phi^0=\phi$, associated with a test configuration and $\phi^t_k$ be the Bergman geodesic ray of Phong-Sturm. Then $L^p$-convergence $\dot{\phi}_k^0 \to \dot{\phi}^0$ holds for any $p\geq1$. 
\end{lem}
\begin{proof}
First we show the $L^1$-convergence. We recall the inequality 
\begin{equation*}
\phi_k^0 -Ck^{-1}\log k \leq \phi^0  \leq \phi_k^0 + Ck^{-1} 
\end{equation*}
again. By the maximality of the solution of the Monge-Amp\`ere equation the left-hand side yields 
$\phi_k^t -Ck^{-1}\log k \leq \phi^t$ for any $t >0$ so that we derive  
\begin{equation*}
\phi^t -\phi^0 \geq \phi_k^t -\phi^0 -C k^{-1}\log k \geq \phi_k^t -\phi_k^0 -Ck^{-1}(1+\log k). 
\end{equation*}
Since both $\phi^t$ and $\phi_k^t$ are convex in $ t$, it follows that 
\begin{equation*}
\dot{\phi}^t \geq \dot{\phi}_k^0 - C\frac{k^{-1}(1+\log k)}{t}. 
\end{equation*} 
Letting $k \to \infty$ first, one has $\dot{\phi}^t \geq \limsup_{k \to \infty} \dot{\phi}_k^0$ and then $\dot{\phi}^0 \geq \limsup_{k \to \infty} \dot{\phi}_k^0$ follows from the $C^{1, \a}$-regularity of the weak geodesic ray.  
On the other hand, the convergence of the integral 
\begin{equation*}
\int_X \dot{\phi}_k^0 \omega^n \to \int_X  \dot{\phi}^0 \omega^n 
\end{equation*}
is well-known fact ({\em e.g.} a very special case of Theorem \cite{His12}). It is then elementary to conclude the $L^1$-convergence by argument of measure theory. Actually one can apply Lemma $2.2$ of \cite{Berm06} directly to our situation. 

Since $\dot{\phi}^0$ and $\dot{\phi}_k^0$ are uniformly bounded (by Lemma $3.1$ of \cite{PS07a}), we may choose a constant $M_p$ such that $\abs{\dot{\phi}^0- \dot{\phi}_k^0}^{p-1} \leq M_p$ holds. Therefore the $L^p$-convergence follows from the estimate: 
\begin{equation*}
\int_X\abs{\dot{\phi}^0- \dot{\phi}_k^0}^p \omega^n = \int_X \abs{\dot{\phi}^0- \dot{\phi}_k^0}^{p-1} \abs{\dot{\phi}^0- \dot{\phi}_k^0}\omega^n
\leq M_p \int_X\abs{\dot{\phi}^0- \dot{\phi}_k^0} \omega^n. 
\end{equation*} 

\end{proof}

We elaborate on the idea of \cite{Don05} and \cite{PS07a} to relate the tangent vector of the weak geodesic ray to endomorphisms. Notice that it gives another proof of the main theorem in \cite{His12} combining with the above lemma.  

\begin{prop}\label{Bergman level}
For any weak geodesic ray $\phi^t$ associated with a test configuration we have 
\begin{equation*}
\frac{1}{V}\int_X (\dot{\phi}_k^0)^p \omega^n = \frac{\Tr (A_k)^p}{k^pN_k} +O(k^{-1}). 
\end{equation*}
Moreover, if the test configuration is $T$-equivariant any $v\in (\ft/\C)_\R$ induces the the trace-free generator $B_k \in  \fsl(H^0(\cX_0, \cL^{\otimes k}_0))$ on the dual such that 
\begin{equation*}
\frac{1}{V}\int_X \dot{\phi}_k^0 u_v \omega^n = \frac{\Tr A_k B_k}{k^2N_k} +O(k^{-1}\log k)
\end{equation*}
holds. 
\end{prop}
\begin{proof}
We prove the second formula. The first one is much easier to show by the same argument. Let us take an equivariant embedding of the test configuration so that  $A_k$ and $B_k$ are identified with the endomorphisms on $H^0(X, L^{\otimes k})$. We may further assume that $A_k$ is Hermitian with respect to $\norm{\cdot}_{k\phi}$. Therefore it can be unitarily diagonalizable by some orthonormal basis $\{s_i\}_{i=1}^{N_k}$. For the basis, the tangent of Phong-Sturm's Bergman construction is computed as 
\begin{equation*}
\dot{\phi}_k^0 = \frac{s^* A_k s}{ks^*s}. 
\end{equation*}
On the other hand, since the test configuration is $T$-equivariant the above $B_k$ is equivalent to what directly induced by $v$ hence we may apply Proposition \ref{product case} to have 
\begin{equation*}
u_v= \frac{s^* B_k s}{ks^*s} +O(k^{-1}\log k). 
\end{equation*}
Setting 
\begin{equation*}
M_{k, ij} := \frac{1}{V} \int_X \frac{s_i\bar{s_j}}{\sum\abs{s_i}^2} \omega^n, 
\end{equation*}
we combine the two in the form:    
\begin{equation*}
\frac{1}{V}\int_X \dot{\phi}_k^0 u_v \omega^n = \frac{\Tr A_kB_k M_k}{k^2} +O(k^{-1}\log k). 
\end{equation*}
Asymptotic of the matrices $M_k$ is computed by Tian-Zelditch-Catlin expansion for the Bergman kernel as follows: 
\begin{align*}
V M_{k, ij} &= \int_X s_i \bar{s_j} e^{-2k\phi_k} \omega^n = \int_X s_i \bar{s_j} e^{-2k\phi+(2k\phi-2k\phi_k)}  \omega^n \\
&= n! k^{-n}\d_{ij} -nk^{n-1} \int_X s_i \bar{s_j} S_{\omega} e^{-2k\phi} \omega^n + O(k^{-n-2}). 
\end{align*}
Here $S_{\omega}$ is the scalar curvature of the K\"ahler metric $\omega$. 

\end{proof} 

\begin{rem}
For general two weak geodesic rays (each of which is associated with a different test configuration), the corresponding generators are not simultaneously diagonalizable hence one cannot describe them in a common orthonormal basis. Also, we fixed in the proof a very special equivariant embedding of a test configuration to identify $A_k$ with a Hermitian operator on $H^0(X, L^{\otimes k})$. To define everything intrinsically on the central fiber, therefore, we need to assume the $T$-equivariance. 
\end{rem} 

Now we fix a $T$-equivariant test configuration $(\cX, \cL)$ to consider the dual  representation $\ft/\C \to \fsl(H^0(\cX_0, \cL_0))$. The image which we denote abuse of notation by $\ft/\C$ then identified with the set of trace-free generators for vector fields. The canonical Killing form over $\fsl(H^0(\cX_0, \cL_0))$ is defined by taking trace of the two endomorphisms' composition and it is non-degenerate. It therefore defines the orthogonal projection $P_k: \fsl(H^0(\cX_0, \cL_0)) \to (\ft/\C)_\R$. Notice that non-degeneracy is not the case for $\ft/\C$.  

\begin{prop}
Let us fix an equivariant embedding of a $T$-equivariant test configuration so that the generator $A_k$ is identified with the operator on $H^0(X, L^{\otimes k})$. Let $\{s_i\}_{i=1}^{N_k}$ be an orthonomal basis which diagonalizes $A_k$. Then the function 
\begin{equation*}
f_k := \frac{s^* P_k(A_k) s}{ks^*s}   
\end{equation*}
converges to $P(\dot{\phi}^0)$ in the $L^p$-space, for any $p\geq1$. 
\end{prop}
\begin{proof}
Since $f_k$ and $P(\dot{\phi}^0)$ are uniformly bounded, it is enough to show $L^2$-convergence. As in the proof of the previous proposition, we have 
\begin{itemize}
\item[$(i)$]
\begin{equation*}
\frac{1}{V}\int_X (f_k)^2 \omega^n = \frac{\Tr P_k(A_k)^2}{k^2N_k} +O(k^{-1}) \ \ \ \text{and}
\end{equation*}
\item[$(ii)$]
\begin{equation*}
\frac{1}{V}\int_X f_ku \omega^n = \frac{\Tr P_k(A_k)B_k}{k^2N_k} +O(k^{-1}\log k), 
\end{equation*}
\end{itemize}
for any $u \in \Ham_0(\ft, \omega)$. Let us take an $L^2$-orthonomal basis $u_j$ $(1 \leq j \leq d)$ of $\Ham_0(\ft, \omega)$. By the definition of $L^2$-orthogonal projection we immediately have 
\begin{equation*}
\frac{1}{V}\int_X P(\dot{\phi}^0)^2 \omega^n = \sum_{j=1}^m \bigg(\frac{1}{V}\int_X \dot{\phi}^0 u_j\omega^n \bigg)^2. 
\end{equation*}
Lemma \ref{Lp} and Proposition \ref{Bergman level} implies 
\begin{equation*}
\frac{1}{V}\int_X \dot{\phi}^0 u_j\omega^n = \lim_{k\to \infty} \frac{\Tr A_k B_k}{k^2N_k}, 
\end{equation*}
so that the linearity yields
\begin{equation*}
\frac{1}{V}\int_X P(\dot{\phi}^0)^2 \omega^n =  \lim_{k\to \infty} \frac{\Tr P_k(A_k)^2}{k^2N_k}. 
\end{equation*}
Then the above formula $(i)$ claims that the $L^2$-norm of $f_k$ converges to that of $P(\dot{\phi}^0)$. 
We may deal with $f_kP(\dot{\phi}^0)$ in a similar manner using  
\begin{equation*}
\frac{1}{V}\int_X f_kP(\dot{\phi}^0) \omega^n =\sum_{j=1}^d \bigg(\frac{1}{V}\int_X \dot{\phi}^0 u_j\omega^n \bigg) \frac{1}{V}\int_X f_k u_j \omega^n.  
\end{equation*}
Indeed the above $(ii)$ concludes that the left-hand side converges to the norm of $P(\dot{\phi}^0)$. We thus have 
\begin{equation*}
\int_X (f_k -P(\dot{\phi}^0))^2 \omega^n \to 0.  
\end{equation*}
\end{proof}

Theorem A is now a direct consequence of the above proposition, if one adopts the same argument as Proposition \ref{Bergman level} using the Tian-Zelditch-Catlin expansion.

\subsection{Definition of the reduced norm}\label{def norm} 

We take over the setting of the previous subsection. Let us set $P^\bot := \id - P$ as the projection to the orthogonal complement. The previous argument shows that the function 
\begin{equation*}
g_k:=  \frac{s^* P^\bot_k(A_k)s}{ks^*s} 
\end{equation*}
converges to $P^\bot(\dot{\phi}^0)$ in $L^p$-space. Thus for a $T$-equivariant test configuration and any integer $p\geq 1$ we obtain the description 
\begin{equation*}
\norm{(\cX, \cL)}_{T, p}^p:= \frac{1}{V}\int_X \abs{P^\bot(\dot{\phi}^0)}^p \omega^n = \lim_{k\to \infty} \frac{\Tr P^\bot_k(A_k)^p}{k^2N_k}. 
\end{equation*}
Notice that in particular the left-hand side is independent of the reference metric $\phi$. 

Let us study about much easier way to define the norm for $T$-equivariant test configurations and compare it with our previous definition. I learned the former in communication with R. Berman and M. Jonsson. When a test configuration $(\cX, \cL)$ is $T$-equivariant, any one-parameter subgroup genearated by $v \in \ft/\C$ induces an action $\mu^{(k)}$ on the central fiber. Then the trace-free parts of the generators $A_k$ and $B_k$, for $\l^{(k)}$ and $\mu^{(k)}$ are simultaneously diagonalizable with eigenvalues $(\l_1, \dots \l_{N_k})$ and $(\mu_1, \dots \mu_{N_k})$. Let us define the norm with respect to $v$ as 
\begin{equation}\label{vnorm}
\norm{(\cX, \cL)}_{v, p}^p := \lim_{k\to \infty} \frac{1}{k^pN_k}\sum_{i=1}^{N_k} \abs{\l_i + \mu_i}^p. 
\end{equation}
The definition further extends to any vector field $v \in (\ft/\C)_\Q$ by homogenuity. In any case the limit exists and described as  
\begin{equation*}
\frac{1}{V}\int_X \abs{\dot{\phi}^0 + u_v}^p \omega^n.  
\end{equation*}
In fact by Hausdorff's theorem it is enough to show the moment convergence  
\begin{equation*}
\frac{1}{V}\int_X (\dot{\phi}^0 + u_v)^p \omega^n =  \lim_{k\to \infty} \frac{\Tr (A_k +B_k)^p}{k^pN_k} 
\end{equation*}
for integral $p\geq 1$. In view of Proposition \ref{Bergman level}, this follows from the $L^p$-convergence $\dot{\phi}_k^0 + u_v \to \dot{\phi}^0 + u_v $ of Lemma \ref{Lp}. One may observe from the analytic description that the norm is continuous in $v$, Therefore the definition extends to $(\ft/\C)_\R$ by density. Finally we show: 

\begin{prop}\label{two def}
Let $T \subseteq \Aut^0(X, L)$ be a sub-torus. Then restricted to $T$-equivariant test configurations, the reduced norm  is equivalent to the infimum norm for any $1\leq p\leq 2$. In fact there exists a positive constant $\d$ such that  
\begin{equation*} 
\d \norm{(\cX, \cL)}_{T, p} \leq  \inf_{v\in (\ft/\C)_\R} \norm{(\cX, \cL)}_{v, p} \leq \norm{(\cX, \cL)}_{T, p} 
\end{equation*}
holds for any $T$-equivariant test configuration. 
\end{prop} 
\begin{proof}
Since $\Ham_0(\ft, \omega)$ is finite dimensitonal, $P^\bot(\dot{\phi}^0)$ can be written to the form $\dot{\phi}+u_v$ for some $v\in (\ft/\C)_\R$. Therefore 
\begin{equation*}
\inf_{v\in (\ft/\C)_\R} \norm{(\cX, \cL)}_{v, p} \leq \norm{(\cX, \cL)}_{T, p} 
\end{equation*}
is immediate. When $p=2$ the converse is also clear. For the rest we may show that there exists a positive constant $C$ such that for any $L^p$-function $f$
\begin{equation*}
 \norm{P^\bot(f)}_p \leq C\norm{f}_p
\end{equation*}
holds. Since $P^\bot(\dot{\phi}^0)=P^\bot(\dot{\phi}^0+u_v)$, we may finish the proof applying the estimate to $f=\dot{\phi}^0+u_v$. By the triangle inequality 
\begin{equation*}
\norm{P^\bot(f)}_p \leq \norm{f}_p + \norm{f-P^\bot(f)}_p =  \norm{f}_p + \norm{P(f)}_p , 
\end{equation*}
it is equivalent to show 
\begin{equation*}
 \norm{P(f)}_p \leq C\norm{f}_p. 
\end{equation*} 
The left-hand side is bounded from above by $\norm{P(f)}_2$ because $X$ is compact. For a fixed orthonormal basis $\{u_j\}_{j=1}^d$ of $\Ham_0(\ft, \omega)$ it is estimated as 
\begin{align*}
\norm{P(f)}_2 &= \bigg[ \sum_{j=1}^d \bigg(\int_X f u_j \omega^n\bigg)^2 \bigg]^{1/2} \\ 
&\leq d \cdot \max_j \abs{\int_X f u_j \omega^n} \leq d \cdot\max_j \norm{u_j}_q \norm{f}_p, 
\end{align*}
where $q\geq 2$ is the H\"older conjugate exponent to $p$. 

\end{proof} 

\begin{dfn}\label{uniform stability relative to G}(\cite{Sze06} for the maximal torus case)
We say that a polarized manifold $(X, L)$ with a fixed subgroup $G \subseteq \Aut^0(X, L)$ is {\em uniformly K-stable relative to $G$}, if there exists a constant $\delta>0$ such that 
\begin{equation}
\DF(\cX, \cL) \geq \delta \norm{(\cX, \cL)}_{C(G), 1} 
\end{equation}
holds for any $G$-equivariant normal test configuration $(\cX, \cL)$. 
\end{dfn} 

When $G$ is trivial, it is equivalent to the usual K-stability condition. The non-uniform version is often called relative K-stability and intensively studied by many authors. In particular, \cite{DSz15} obtained a K\"ahler-Einstein metric from this non-uniform condition.   
We expect that for a reductive subgroup $G=K_\C \subseteq  \Aut^0(X, L)$ our uniform K-stability relative to $G$ is equivalent to the coercivity estimate 
\begin{equation}\label{C(G)-coercivity}
M(\phi) \geq \delta \inf_{g \in C(G)} J(\phi_g) -C,     
\end{equation} 
for any positively curved $K$-invariant metric $\phi$. It was shown by \cite{DR15} that (\ref{C(G)-coercivity}) in fact holds for K\"ahler-Einstein Fano manifolds when the center of $G=\Aut^0(X, L)$ is trivial. Also in this analytic side the restriction to $K$-invariant metrics does not lose the generality because any constant scalar curvature K\"ahler metric is invariant for some maximal compact subgroup, by the theorem of Matsushima-Lichnerowicz.  
Finally, we notice that there exists another coercivity condition which requires 
\begin{equation}\label{G-coercivity} 
M(\phi) \geq \delta \inf_{g \in G} J(\phi_g) -C     
\end{equation} 
for any positively curved metric. The latter coercivity holds for general constant scalar curvature K\"ahler manifolds, as it was very recently shown by \cite{BDL16}. 

\subsection{Examples}\label{examples}

\begin{exam}[Toric polarization]

Let $(X, L)$ be a toric polarized manifold, defined by a moment polytope $P$ on the lattice $M$.  Here $G/\C^*$ is chosen to be the maximal torus $\T$. 
In this case, a $\T$-equivariant test configuration is equivalent to so-called toric test configuration and it is represented by a rational piecewise-linear convex funtion $f$ on $P$. The class of test configurations was intensively studied in \cite{Don02}. 
Each vector field  $v \in \ft/\C$ is given by the affine function $\ell$ on the dual of $(\ft/\C)_\R$. 
Following \cite{Don02} one can compute to check 
\begin{equation*}
\norm{(\cX, \cL)}_{v, p} = \bigg( \frac{1}{V}\int_P \abs{ f +\ell}^p \bigg)^{1/p}.  
\end{equation*}
See also \cite{WN10} for general case in terms of the Okounkov body. 
After all Proposition \ref{two def} shows that our introducing reduced norm is equivalent to G. Sz\'ekelyhidi's original definition for the toric setting. 
\end{exam}

\begin{exam}[Deformation to the normal cone]
Any deformation to the normal cone of $G$-invariant ideal $\cJ$ (or the associated subscheme $V$) defines a $G$-equivariant test configuration. The space $\cX$ is given as the normalized blow-up of $X \times \C$ by the ideal $\cJ +\tau \cO$. It is endowed with the $\C^*$-action trivial on the $\cJ$-blow up $X'$ of $X$. The action of $\alpha \in \C^*$ to the exceptional divisor $P=\P(N_{V/X} \oplus  \cO_V)$ is given by $\begin{pmatrix} 1 & 0 \\  0 & \alpha \end{pmatrix}$. For each $g \in G$ the restriction of the induced $G$-action on $\cX$ to $P$ is written as $\begin{pmatrix} g_* & 0 \\  0 & 1  \end{pmatrix}$. 
Then for a sufficiently small $c\in \Q_{>0}$, $\cL:=p^*L-cP$ defines an ample $\Q$-line bundle so that the construction produces a large amount of $G$-equivariant test configurations. 

For any 1-PS to the center, the same family $(\cX, \cL)$ with the action to $P$ by $\begin{pmatrix} \mu(\alpha)_* & 0 \\  0 & \alpha \end{pmatrix}$ defines a $G$-equivariant test configuration $(\cX_\mu, \cL_\mu)$. 


For example, $\P^2$ has no non-trivial $GL(3, \C)$-invariant ideal so that it does not admit any deformation to the normal cone. There actually exists no $GL(3, \C)$-equivariant (normal) test configuration and $\P^2$ is uniformly K-stable relative to $GL(3, \C)$.  

\end{exam}


\section{Characterization of product test configurations}\label{characterization}

In this section we assume that the reduced norm is zero, to show the productness of a $T$-equivariant test configuration. 
By the definition of the reduced norm it implies that $\dot{\phi}^0 \in \Ham_0(\ft, \omega)$. Let us denote the associated generator by $B_k \in (\ft/\C)_\R \subseteq  \fsl(H^0(\cX_0, \cL_0^{\otimes k}))$. The problem here is that we do not a priori know whether the generator $A_k$ of $(\cX, \cL)$ coincides with $B_k$. 

\begin{lem}
The Hamilton function $\dot{\phi}^0$ corresponds to an element of $(\ft/\C)_\Q$. 
\end{lem}
\begin{proof}
This is precisely the rationality of $P(\cX, \cL)$ which we mentioned in the introduction. 
It is essentially the rationality of the Riemann-Roch coefficients and the following argument seems familiar one to the experts, but I would like to write down the detail for the convenience to readers. 
We denote the the Hamilton vector field for $\dot{\phi}^0$ by $v \in (\ft/\C)_\R$. 
We may take an orthogonal basis $v_1, \dots, v_d$ for $(\ft/\C)_\R$ such that each $v_j$ $(1 \leq j \leq d)$ derives from a 1-PS. The associated generators $B_{j, k} \in \fsl(H^0(\cX_0, \cL_0^{\otimes k}))$ are commutative to each other. Let us denote the real parts of the corresponding normalized Hamilton functions by $u_1, \dots, u_d$.  

As we took $B_{j, k}$ from some 1-PS, the right-hand side of the formula 
\begin{equation*} 
\frac{1}{V}\int_X \dot{\phi}^0 u_j \omega^n = \lim_{k\to \infty} \frac{\Tr A_k B_{j, k}}{k^2N_k}  
\end{equation*}
can be described as a characteristic number hence it is a rational number for $1 \leq j \leq d$. We must use $A_k$ instead of $B_k$ which is a priori not rational. Actually one can apply the equivariant Hirzebruch-Riemann-Roch formula (see \cite{EG98} and \cite{EG00}, or \cite{BHJ15}, Appendix B) to $A_k +B_{j, k}$, $A_k$, and $B_{j, k}$, each of which derives from a $\C^*$-action to $(\cX_0, \cL_0)$. For the induced action to $H^0(\cX_0, \cL_0^{\otimes k})$ sum of the weights like $\Tr A_k$ appears in the degree two term in the equivariant cohomology $H^*_{\C^*}(\pt; \Z) \simeq \Z[t]$. The square sum of the weights like $\Tr A_k^2$ appears in the degree four term. Finally the above trace can be expressed as 
\begin{equation*}
2\Tr A_k B_{j, k} = \Tr(A_k + B_{j, k})^2 - \Tr A_k^2 -\Tr B_{j, k}^2. 
\end{equation*}
The right-hand side is a polynomial in $k$, which has integer coefficients. In conclusion, $\dot{\phi}^0$ is a $\Q$-linear combination of $u_1, \dots u_d$. Since each of the vector fields $v_1, \dots v_d$ derives from a 1-PS and is attached to the torus, we conclude the claim. 

\end{proof} 

Now we may assume that $\dot{\phi}^0$ is further corresponds to a 1-PS to $T$, by replacing $(\cX, \cL)$ with its base change by $\tau \mapsto \tau^d$. In what follows we shall prove $A_1=B_1$. We may as well assume that $\cL$ is sufficiently relatively ample after the line bundle is replaced with its high tensor power. We then set $A:=A_1$ and $B:=B_1$. 

Let us consider the test configuration $(\cX', \cL')$ which is isomorphic to $(\cX, \cL)$ as a family but the $\C^*$-action is given by $A-B$. In fact, for any fixed equivariant embedding of $(\cX, \cL)$ the same $(\cX', \cL')$ is constructed by taking the Zariski closure of the orbit in $\P(H^0(\cX_0, \cL_0)^\vee)$. Moreover the generator for $\fsl(H^0(\cX'_0,  (\cL'_0)^{\otimes k}))$ is given by $(A-B)_k=A_k-B_k$. 
At the end we compute the (non-reduced) norm of this test configuration. It decomposes into 
\begin{equation*}
\norm{(\cX', \cL')}_2^2 = \lim_{k \to \infty} \frac{\Tr (A_k-B_k)^2}{k^2N_k}
= \lim_{k \to \infty} \frac{\Tr A_k^2 -2\Tr A_k B_k -\Tr B_k^2}{k^2N_k}. 
\end{equation*} 
For the first term we already have 
\begin{equation*}
\int_X (\dot{\phi}_k^0)^2 \omega^n = \frac{\Tr A_k^2}{k^2N_k} +O(k^{-1}). 
\end{equation*} 
Since $\dot{\phi}^0$ is the real part of a Hamilton function, the second term has the expression: 
\begin{equation*}
\int_X \dot{\phi}_k^0\dot{\phi}^0 \omega^n = \frac{\Tr A_k B_k}{k^2N_k} +O(k^{-1}\log k). 
\end{equation*} 
Finally the equivariant Chern-Weil theory translates the third term as 
\begin{equation*}
\int_X (\dot{\phi}^0)^2 \omega^n = \lim_{k \to \infty} \frac{\Tr B_k^2}{k^2N_k}. 
\end{equation*} 
Combining these with Lemma \ref{Lp} we deduce $\norm{(\cX', \cL')}_2=0$. It then follows from \cite{BHJ15} that $(\cX', \cL')$ is trivial. In particular $A_1=B_1$ holds. 

\qed



\newpage
\begin{thebibliography}{widestlabel}

       
\bibitem[AB84]{AB84}M.~F. Atiyah and R.~Bott:
       \newblock \emph{The moment map and equivariant cohomology}. 
       \newblock Topology 23 (1984), no. \textbf{1}, 1--28. 



\bibitem[Berm06]{Berm06}R. Berman: 
      \newblock \emph{Super Toeplitz operators on line bundles}. 
      \newblock J. Geom. Anal. \textbf{16} (2006), no. 1, 1--22.



\bibitem[Berm12]{Berm12}R.~J. Berman: 
   \newblock \emph{K-polystability of Q-Fano varieties admitting Kahler--Einstein metrics}. 
   \newblock \texttt{arXiv:1205.6214}. 



\bibitem[BBJ15]{BBJ15}R.~J. Berman, S. Boucksom, and M.~Jonsson: 
    \newblock \emph{A variational approach to the Yau-Tian-Donaldson conjecture}. 
    \newblock \texttt{arXiv:1509.04561}. 

\bibitem[BDL16]{BDL16}R.~J. Berman, T.~Darvas, and C.~H. Lu: 
    \newblock \emph{Regularity of weak minimizers of the K-energy and applications to properness and K-stability}. 
   \newblock \texttt{arXiv:1602.03114}. 
 


   
\bibitem[Bern09]{Bern09}B.~Berndtsson: 
    \newblock \emph{Probability measures related to geodesics in the space of K\"{a}hler metrics.}   
    \newblock \texttt{arXiv:0907.1806}. 





\bibitem[BEGZ10]{BEGZ10}S.~Boucksom, P.~Eyssidieux, V.~Guedj, and A.~Zeriahi: 
   \newblock \emph{Monge--Amp\`{e}re equations in big cohomology classes.}
   \newblock   Acta Math. \textbf{205} (2010), no. 2, 199--262. 






\bibitem[BHJ15]{BHJ15}
S.~Boucksom T.~Hisamoto and M.~Jonsson: 
\newblock \emph{Uniform K-stability, Duistermaat-Heckman measures and singularities of pairs}. 
\newblock \texttt{arXiv:1504.06568}.   

\bibitem[BHJ16]{BHJ16}
S.~Boucksom T.~Hisamoto and M.~Jonsson: 
\newblock \emph{Uniform K-stability and asymptotics of energy functionals in K\"ahler geometry}. 
\newblock \texttt{arXiv:1603.01026}. 



    


\bibitem[CS16]{CS16}G. Codogni and J. Stoppa: 
    \newblock \emph{Torus equivariant K-stability}. 
    \newblock \texttt{arXiv: arXiv:1602.03451}. 

\bibitem[CT08]{CT08}X. Chen and Y. Tang: 
    \newblock \emph{ Test configuration and geodesic rays}. 
    \newblock Ast\'{e}risque No. \textbf{321} (2008), 139--167.

\bibitem[DR15]{DR15}T.~Darvas and Y.~A. Rubinstein: 
    \newblock \emph{Tian's properness conjectures and Finsler geometry of the space of K\"ahler metrics}. 
    \newblock \texttt{arXiv:1506.07129}. 


\bibitem[Der14a]{Der14a} 
R.~Dervan: 
\newblock \emph{Uniform stability of twisted constant scalar curvature K\"ahler metrics}. 
\newblock \texttt{arXiv:1412.0648}. 
\newblock Int. Math. Res. Notices \textbf{15} (2016), 4728--4783.  

\bibitem[Der14b]{Der14b}
R.~Dervan: 
\newblock \emph{Alpha invariants and coercivity of the Mabuchi functional on Fano manifolds.}
\newblock Ann. Fac. Sci Toulouse Math. (6) \textbf{25} (2016) no. 4, 919--934. 

 
 
 
\bibitem[Don02]{Don02}S.~K. Donaldson: 
   \newblock \emph{Scalar curvature and stability of toric varieties}. 
   \newblock  J. Differential Geom. \textbf{62} (2002), no. 2, 289--349.  
 
\bibitem[Don05]{Don05}S.~K. Donaldson: 
   \newblock \emph{Lower bounds on the Calabi functional}. 
   \newblock  J. Differential Geom. \textbf{70} (2005), no. 3, 453--472.
   

\bibitem[DSz15]{DSz15} 
V.~Datar and G.~Sz\'ekelyhidi.
\newblock\emph{K\"ahler-Einstein metrics along the smooth continuity method}. 
\newblock \texttt{arXiv:1506.07495}. 

\bibitem[EG98]{EG98} 
D.~Edidin and W.~Graham.
\newblock \emph{Equivariant intersection theory.}
\newblock Invent. Math. {\bf 131} (1998), no. 3, 595--634. 

\bibitem[EG00]{EG00} 
D.~Edidin and W.~Graham.
\newblock\emph{Riemann-Roch for equivariant Chow groups.}
\newblock Duke Math. J.{\bf 102} (2000), no. 3, 567--594.

\bibitem[FM95]{FM95}A.~Futaki and T.~Mabuchi: 
        \newblock \emph{Bilinear forms and extremal K\"{a}hler vector fields associated with K\"{a}hler classes.} 
        \newblock Math. Ann. \textbf{301} (1995), no. 2, 199--210.  
   
   
\bibitem[His12]{His12}T.~Hisamoto: 
   \newblock \emph{On the limit of spectral measures associated to a test configuration of a polarized K\"{a}hler manifold}.  
   \newblock \texttt{arXiv:1211.2324}, to appear in J. Reine Angew. Math..
   



\bibitem[LX11]{LX11}C. Li and C. Xu: 
   \newblock  {\em Special test configurations and $K$-stability of Fano varieties.}
   \newblock Ann. of Math. \textbf{180} (2014) no. 1, 197--232. 


\bibitem[Mab90]{Mab90}T.~Mabuchi: 
   \newblock \emph{An algebraic character associated with the Poisson brackets}. 
   \newblock Recent topics in differential and analytic geometry, 339--358, Adv. Stud. Pure Math., \textbf{18}-I, Academic Press, Boston, MA, 1990. 







       





    
\bibitem[PS07a]{PS07a}D.~H. Phong and J.~Sturm: 
   \newblock \emph{Test configurations for K-stability and geodesic rays}. 
   \newblock J. Symplectic Geom. \textbf{5} (2007), no. 2, 221--247.

\bibitem[PS07b]{PS07b}D.~H. Phong and J.~Sturm: 
   \newblock \emph{On the regularity of geodesic rays associated to test configurations}. 
   \newblock \texttt{arXiv:0707.3956}. 

\bibitem[PS10]{PS10}D.~H. Phong and J.~Sturm:
    \newblock \emph{Regularity of geodesic rays and Monge-Amp\`{e}re equations}. 
    \newblock Proc. Amer. Math. Soc. \textbf{138} (2010), no. 10, 3637--3650.



\bibitem[RWN11]{RWN11}J.~Ross and D.~Witt Nystr\"{o}m: 
   \newblock \emph{Analytic test configurations and geodesic rays.}
   \newblock J. Symplectic Geom. \textbf{12} (2014), no. 1, 125--169. 

\bibitem[Sem92]{Sem92}S.~Semmes: 
    \newblock \emph{Complex Monge--Amp\`{e}re and symplectic manifolds.}
    \newblock Amer. J. Math. 114 (1992), no. \textbf{3}, 495--550.



\bibitem[Sz\'e06]{Sze06}G.~Sz\'ekelyhidi: 
    \newblock  \emph{Extremal metrics and K-stability.}
    \newblock \texttt{arXiv:0611002}. Ph.D Thesis.
    




\bibitem[WN10]{WN10}D. Witt Nystr\"{o}m: 
     \newblock \emph{Test configurations and Okounkov bodies.}
     \newblock Compositio Math. \textbf{148} (2012), 1736--1756. 





\end{thebibliography}
\end{document}